\documentclass[12pt,twoside,reqno]{amsart}

\usepackage{hyperref}
\usepackage{amsthm, amsmath, amscd, amssymb,centernot}
\usepackage[all]{xy}
\usepackage[T1]{fontenc}
\usepackage[left=2.5cm,top=2.5cm,bottom=3cm,right=2.5cm]{geometry}
\usepackage{amsthm, amsmath, amscd, amssymb,centernot}
\usepackage{comment}
\usepackage{tikz}
\usepackage{csquotes}

\newcommand{\s}{\varepsilon}

\numberwithin{equation}{section}
\newtheorem{theorem}{Theorem}[section]
\newtheorem*{theorem*}{Theorem}

\newtheorem{definition}[theorem]{Definition}
\theoremstyle{definition}

\newtheorem{remark}[theorem]{Remark}
\newtheorem{example}[theorem]{Example}

\begin{document}
\title{Some Results on Seshadri constants on Surfaces of general type}
\author[Praveen Kumar Roy]{Praveen Kumar Roy}
\address{Chennai Mathematical Institute, H1 SIPCOT IT Park, Siruseri, Kelambakkam 603103, India}
\email{praveenkroy@cmi.ac.in}

\thanks{Author was partially supported by a grant from Infosys Foundation}

\date{May 24, 2019}
\maketitle
\begin{abstract}
We prove two new results for Seshadri constants on surfaces of 
general type. Let $X$ be a surface of general type. In the first part, 
inspired by \cite{B-S}, we list the possible values for the multi-point 
Seshadri constant $\s(K_X,x_1,x_2,...,x_r)$ when it lies between 
$0$ and $1/r$, where $K_X$ is the canonical line bundle on $X$. 
In the second part, we assume $X$ of the form $C \times C$, 
where $C$ is a general smooth curve of genus $g \geq 2$. 
Given such $X$ and an ample line bundle $L$ on $X$ with some conditions 
on it, we show that the global Seshadri constant of $L$ is a rational number.
\end{abstract}

% For many users, the previous commands will be enough.
% If you want to directly input Unicode, add an Input Menu or Keyboard to the menu bar 
% using the International Panel in System Preferences.
% Unicode must be typeset using a font containing the appropriate characters.
% Remove the comment signs below for examples.

\section{introduction}\label{intro}

Seshadri constants have turned out to be a powerful tool to measure local positivity of an 
ample line bundle on a projective variety. They were defined by Demailly using the 
Seshadri criterion of ampleness for a line bundle \cite{D-JPS}. Since then, the area 
has emerged to be quite active with computing and bounding the Seshadri constants 
becoming an active area of research. For a detailed survey and the typical nature of work, 
we refer to \cite{primer,D-K-M-S,F,K-A,K-B,K-P}.

Let $X$ be a smooth complex projective surface and let $L$ be a line bundle on $X$. 
Given a point $x \in X$, Seshadri criterion for ampleness \cite{H} says that $L$ is an 
ample line bundle on $X$ if and only if there exists a positive real number $\s $ 
such that $L\cdot C \geq \s \cdot mult_x C $ for all $x\in C$. Here, ``$mult_x C$" denotes the 
multiplicity of $C$ at the point $x$.  Given $X$ and $L$ as above, the {\it Seshadri constant} 
$\varepsilon(X,L,x)$ of $L$ at a point $x\in X$ \cite{L} is defined as
\[
\varepsilon(X,L,x):=  \inf\limits_{\substack{x \in C}} \frac{L\cdot  C}{{\rm mult}_{x}C},\
\]
where the infimum is taken over all irreducible and reduced curves $C$ passing through $x$. 
Now, it is easy to see that $L$ is ample if and only if $\s(X,L,x) > 0$ for all $x\in X$. There are 
various directions in which one can study Seshadri constants. For a comprehensive survey, 
we refer to \cite{primer}. 

Given a smooth complex projective surface $X$ and an ample line bundle $L$ on $X$, it is not difficult 
to see that $0< \s(X,L,x) \leq \sqrt{L^2}$ for every $x\in X.$ Thus, it makes sense to define:
\[
\s(X,L,1) := \sup\limits_{x\in X}\s(X,L,x) \quad \text{and} \quad \s(X,L) := \inf\limits_{x\in X}\s(X,L,x).
\]
These satisfy the following inequalities:
\[
0< \s(X,L) \leq \s(X,L,x) \leq \s(X,L,1) \leq \sqrt{L^2}.
\]

It is known that $\s(X,L,1)$ is attained at a very general point $x\in X$ \cite{O}. Further, 
if $\s(X,L,1) < \sqrt{L^2} $, then there exists a reduced and irreducible curve $C\subset X$ 
containing a very general point $x$ such that $\s(X,L,1) = \s(X,L,x) = \frac{L\cdot C}{mult_xC}$. 
Therefore, the Seshadri constant is a rational number in this situation. Consequently, for an 
irrational Seshadri constant, $\s(X,L,x)$ must be equal to $\sqrt{L^2}$ and, $L^2 $ must 
be non-square. However, there is no known example of a triple $(X,L,x)$ which gives an irrational 
Seshadri constant.

On the other hand, $\s(X,L)$ is computed at some special points $x\in X$. Therefore, in order to 
compute $\s(X,L)$ one needs to find some information about the curves passing through $x\in X$. 
In other words, geometry of $X$ near that point is important.

\subsection{Multi-point Seshadri constants}

Let $X$ be a smooth complex projective surface and $L$ be an ample line bundle on $X$. 
Let $r\geq 1$ be an integer and $x_1,x_2,...,x_r \in X$ be $r$ distinct points. Then, the 
multi-point Seshadri constant of $L$ at $x_1,x_2,...,x_r \in X$ is defined as:
\begin{eqnarray*}
\s(X,L,x_1,x_2,...,x_r) := \inf \limits_{C \cap \{x_1, x_2,...,x_r \}\neq \phi }\frac{L\cdot C}{\sum\limits_{i=1}^{i=r}{mult_{x_i}C}},
\end{eqnarray*}
where the infimum is taken over all reduced and irreducible curves $C \subset X$ passing 
through at least one of the points $x_1,x_2,...,x_r \in X$. A well known upper bound for the 
multi-point Seshadri constant is 
\begin{eqnarray*}
\s(X,L,x_1,x_2,...,x_r) \leq \sqrt{\frac{L^2}{r}}.
\end{eqnarray*}
The Seshadri constant is said to be {\it sub-maximal} if the above inequality is strict, and in 
that case it is computed by a curve $C\subset X$ (i.e., $\s(X,L,x_1,x_2,...,x_r) = 
L\cdot C/\sum\limits_{i=1}^{i=r}{mult_{x_i}C}$),  which is known as the {\it Seshadri curve}. 
See \cite[Proposition 1.1]{B-S} for a proof of their existence in the single point case which 
generalizes easily to the multi-point case. 

One can then define:
\[
\s(X,L,r) := \max\limits_{\substack{x_1,x_2,...,x_r \in X}}\s(X,L,x_1,x_2,...,x_r).
\]
It is known that $\s(X,L,r)$ is attained at very general points $x_1,x_2,...,x_r \in X$, 
i.e., there exists a set $U\subset X^r :=X\times X\times ...\times X$ which is 
the complement of a countable union of proper closed subsets of $X^r$, such that 
$\s(X,L,r) = \s(X,L,x_1,x_2,...,x_r)$ for all $(x_1,x_2,...,x_r) \in U$.
It is conjectured that, $\s(X,L,r)$ is equal to $\sqrt{{L^2}/{r}}$ for large $r$ \cite{S-Gas}. In fact, 
the {\it Nagata-Biran-Szemberg Conjecture} predicts exactly when it happens. It says that the multi-point Seshadri 
constant at a very general set of points is maximal when $r\geq k_0^2L^2$, where $k_0$ 
is the smallest integer such that the linear system $|k_0L|$ contains a smooth non-rational curve.   

%\begin{eqnarray*}
%\s(X,L,x_1,x_2,...,x_r) \leq \sqrt{\frac{L^2}{r}}.
%\end{eqnarray*} 

%In this article we have considered the surface of general type of the form $C\times C$, where $C$ is a general member of the moduli of smooth curves of genus at least 2 and try to .

In this article, we study some of the questions discussed above on surfaces of general type. 
Note that surfaces of general type \cite{B} are minimal surfaces of Kodiara dimension two 
(see def. \eqref{Kodaira dimension}). Not much is known about these surfaces compared to 
surfaces of Kodaira dimension $1$, $0$ or $-\infty$. Here, we have considered a class of 
such surfaces of the form $C \times C$, where $C$ is a smooth curve of genus at least two, 
and have answered some of the questions about Seshadri constants.

This paper is divided into two parts. In \S \eqref{multi-point}, we prove a result about the multi-point 
Seshadri constant of canonical line bundle on a surface of general type. In \S \eqref{CxC}, we consider 
surfaces of general type of the form $C \times C$, where $C$ is a general member of the moduli 
of smooth curves of genus $g\geq 2$ and answer some of the questions discussed above.

\section{Multi-point Seshadri constants on surfaces of general type}\label{multi-point}

Let $X$ be a smooth complex projective variety, and let $L$ be a line bundle on $X$. 
Consider the linear system $|mL|$ for $m\in \mathbb{N}$. The global sections of $mL$ 
defines a rational map  
\begin{eqnarray*}
\phi_{mL} : X \dasharrow \mathbb{P}(H^0(X, mL)).
\end{eqnarray*}

Clearly $dim(\phi_{mL}(X)) \leq dim(X)$. 

\begin{description}
\item [Notation] $\kappa(X,L) := \max\{dim(\phi_{mL}(X)) : m\in \mathbb{N} \}$.
\end{description}

\begin{definition}\label{Kodaira dimension}
Given a smooth complex projective variety $X$ with canonical divisor $K_X$, 
the Kodaira dimension of $X$ is defined as $\kappa(X,K_X)$.
\end{definition}

\begin{definition}
A smooth complex algebraic surface $X$ is said to be of general type if the Kodaira 
dimension $\kappa(X) = 2$. %\cite{B}.  
\end{definition}

One defines a line bundle $L$ on a smooth complex projective variety $X$ to 
be {\it big} if $\kappa(X,L) = dim(X)$. Therefore, a surface of general type is a 
surface whose canonical divisor is big. The following theorem is a characterisation 
for a nef line bundle to be big \cite{L}. 

\begin{theorem}
Let $X$ be an irreducible projective variety of dimension n and $L$ be a nef line 
bundle on $X$. Then $L$ is big if and only if its top self-intersection is strictly positive, i.e., $(L^n) > 0$.
\end{theorem}

%\section{Seshadri constants}
Motivated by \cite[Theorem 1]{B-S}, we prove the following: 
\begin{theorem}
Let $X$ be a surface of general type and $K_X$ be the canonical line bundle on $X$. 
If $K_X$ is big and nef and $x_1,x_2,...,x_r \in X$ are $r\geq 2$ points, then we have the following.
\begin{enumerate}
\item  $\s(X,K_X,x_1,x_2,...,x_r) = 0  \Leftrightarrow$ at least one of $x_i$ lies on one of the finitely many (-2)-curves on $X$.

\item  If $0< \s(X,K_X,x_1,x_2,...,x_r) < \frac{1}{r}$, then 
	\begin{equation*}
	\s(X,K_X,x_1,x_2,...,x_r) = 
		\begin{cases} 
		\frac{1}{r+1} \; \text{or} \;\frac{2}{5} & \text{if} \quad r = 2, \cr
		\frac{1}{r+1} \; \text{or} \; \frac{1}{r+2}  & \text{if} \quad 3 \leq r < 10, \cr
		\frac{1}{r+1} \; \text{or} \; \frac{1}{r+2} \; \text{or} \; \frac{1}{r+3} & \text{if} \quad r \geq 10. 
		\end{cases}
	\end{equation*}
\end{enumerate}
\end{theorem}

\begin{proof}
The proof of $(1) $ uses the same technique as the case of the single point 
Seshadri constant in \cite{B-S}. Since $K_X$ is big and nef, its self-intersection 
is strictly positive, i.e., $K_X ^2 > 0$. 

$\Rightarrow$: Let $C\subset X$ be a smooth curve passing through at least one %(Is it necessary to put $\Rightarrow$?)
of the points $x_1,x_2,...,x_r \in X$ with multiplicities $m_1,m_2,...,m_r$, such that 
$0 = \s(X,K_X,x_1,x_2,...,x_r) = K_X\cdot C/ \sum\limits_{i=1}^{i=r}{m_i}$. 
This gives $K_X\cdot C = 0$. Using the Hodge Index Theorem and the fact that 
$K_X^2 >0$, we get $C^2 < 0$. Since $K_X$ is nef, there are no $(-1)$-curves in $X$, 
therefore $C^2 = -2$. Using adjunction formula we conclude that the genus of $C$ is 
$0$, and hence $C$ is a rational curve.

$\Leftarrow$: Conversely, suppose some $x_i$ lies on a $(-2)$-curve $C$, then using 
the adjunction formula and the fact that the arithmetic genus of $C$ is $0$, we find that $K_X\cdot C = 0$. 
Hence $\s(X,K_X,x_1,x_2,...,x_r) = \frac{K_X\cdot C}{\sum\limits_{i=1}^{i=r}{mult_{x_i}C}} = 0$. 

$(2)$ Let $\s(X,K_X,x_1,x_2,...,x_r) < \frac{1}{r}$ which in turn is less than 
$\sqrt{\frac{K_X^2}{r}}$, so that by a generalized statement of \cite[Proposition 1.1]{B-S} 
for the multi-point case, there exists a reduced and irreducible curve $C$ computing 
$\s(X,K_X,x_1,x_2,...,x_r).$  Let $C$ be a reduced and irreducible curve in $X$ passing 
through at least one of the points $x_1,x_2,...,x_r$ with multiplicities $m_1,m_2,...,m_r$ 
such that 
\begin{equation}
\s(X,K_X,x_1,x_2,...,x_r) = \frac{K_X\cdot C}{m}, \label{equation1}
\end{equation}
where $m := \sum\limits_{i=1}^{i=r}{m_i}$. Put $K_X\cdot C = d$. Notice that 
$\frac{d}{m}<\frac{1}{r}$ which gives $m>dr$. Now, using the positivity of $K_X^2$ 
and the Hodge Index Theorem, we get $C^2 \leq K_X^2C^2\leq (K_X\cdot C)^2 = d^2$. 
Let $\tilde{C}$ be the normalization of $C$. Then,

\begin{eqnarray} \nonumber
0 \leq P_a(\tilde{C}) &\leq& P_a(C) - \sum\limits_{i=1}^{i=r}\frac{m_i(m_i -1)}{2} \\ \nonumber
\Rightarrow \sum\limits_{i=1}^{i=r}\frac{m_i(m_i -1)}{2} &\leq& P_a(C) = 
1+ \frac{1}{2}C^2+ \frac{1}{2}K_X\cdot C \leq 1+ \frac{d(d+1)}{2} \\ \nonumber
\Rightarrow \left(\frac{1}{r}m^2 -m\right) &\leq& 
\left(\sum\limits_{i=1}^{i=r}{m_i}^2 -\sum\limits_{i=1}^{i=r}{m_i}\right) < 
2+d^2 +d \\ \label{equa2}
\Rightarrow \quad m^2 -rm &-&r(2+d^2+d) < 0. \\ \nonumber
\end{eqnarray}

We see that equation \eqref{equation1} implies the inequality \eqref{equa2}. Therefore, 
it is enough to find out when the inequality \eqref{equa2} holds. We show that the 
possible choices of $d$ and $m$ satisfying the above conditions are as stated in 
the statement of the theorem. 

Put $\phi_{r,d}(m) := m^2 -rm -r(2+d^2+d)$. 
\begin{description}
\item [Claim] $\phi_{r,d}(m)<0$ $\Longrightarrow$ $d=1,2$ and $m=r+1,r+2,r+3\;\text{and}\;5$ 
with some conditions on $r$.
\end{description}
Since $m>dr$, substituting $m =dr+j$ in $\phi_{r,d}(m)$, we get 
\[ \phi_{r,d}(dr+j) := r^2d^2 +j^2 + 2drj  - r^2d -rd^2 -(d+2)r -rj. \]
\begin{description}
\item [$d=1$] \quad $\phi_{r,1}(r+j) = j^2 +rj -4r \; \begin{cases} 
<0 & \text{if $j=1$ and $r\geq2$} \cr
<0 & \text{if $j=2$ and $r\geq3$}\cr
<0 & \text{if $j=3$ and $r\geq10$}\cr
>0 & \text{otherwise}\cr
\end{cases}
$ 
\item [$d=2$] \quad $\phi_{r,2}(2r+j) = 2r^2+j^2 +3rj - 8r \; \begin{cases}
<0 & \text{if $j=1$ and $r=2$}\cr
>0 & \text{otherwise}\cr
\end{cases}
 $ 
\item [$d\geq3$] \quad $\phi_{r,d}(rd+j) > 0$ for $r\geq 2$.
\end{description}

In order to see the last statement, it is sufficient to show that $\phi_{r,3}(3r+j) >0$ for 
$r\geq2$ and the derivative of $\phi_{r,d}$ with respect to $d$ is positive. This will imply 
that $\phi_{r,d}$ is an increasing function of $d$ and hence positive for all $d\geq3, r\geq2$. 
The first condition is easily checked. The second condition is also satisfied since differentiating 
$\phi_{r,d}$ with respect to $d$ gives
\[ \phi'_{r,d}(rd+j) = r^2(2d-1) - r(2d+1-2j)\]
which is always positive whenever $r\geq2$.
\end{proof}

\section{Surface of general type of the form $C\times C$}\label{CxC}

Let $C$ be a smooth complex projective curve of genus $g \geq 2$ and 
consider a surface $X= C\times C$. Let $F_1 $ and $F_2$ be fibres 
corresponding to the two projections from $C\times C \longrightarrow C$ 
and let $\delta$ be the diagonal. Assume that $C$ is a general member of 
the moduli of smooth curves of genus $g$, where $g\geq 2$. Then, it is 
known that the Néron-Severi group $NS(X)$ is spanned by $F_1, F_2$ 
and $\delta$ \cite[1.5B]{L}. Intersections among them is governed by 
the following formulae:
\begin{eqnarray*}
 (F_1)^2 &=& 0, \qquad
 (F_2)^2 = 0, \cr
 F_1 \cdot F_2 &=& F_1 \cdot \delta = F_2 \cdot \delta =1, \cr
\text{and}\quad &\delta^2& = 2-2g. \nonumber
\end{eqnarray*}
Let $K_X$ be the canonical divisor of $X$. Then, it can be checked that 
$K_X^2 = 8(g-1)^2$ is always positive \cite{H}.

We consider $X$ defined as above and compute the Seshadri constant of an ample 
line bundle $L$ on $X$. Let $L \equiv_{num} a_1 F_1 + a_2 F_2 + a_3 \delta$, 
where $a_1 ,a_2, a_3 \in \mathbb{Z}$ and \enquote{$\equiv_{num}$} represents the 
$numerical$ $equivalence$. Since $L$ is ample, we have

\begin{eqnarray}
L\cdot F_1 &=& a_2 + a_3 > 0, \cr
L\cdot F_2 &=& a_1 + a_3 > 0, \cr
L\cdot \delta &=& a_1 + a_2 - (2g-2)a_3 > 0,\quad \text{and} \cr
L^2 &=& 2(a_1 a_2 + a_1 a_3 + a_2 a_3)- (2g-2)a_3 ^2 > 0. 
\end{eqnarray}

\subsection{Results about $\s(L)$}\label{special}

In this section we partially answer the question about the rationality of 
$\s(L)$ \cite[Question 1.6]{S-S-B-S}. In other words, under some conditions 
on $a_1,a_2$ and $a_3$ we address the question of rationality in affirmative. 
Following is our main theorem.

\begin{theorem}
Let $X=C\times C$, where $C$ is a general member of moduli of smooth curves of genus $g\geq 2$. 
Let $L \equiv_{num} a_1F_1+a_2F_2+a_3\delta$ be an ample line bundle satisfying any of the 
following conditions on $a_1,a_2$ and $a_3$.
\begin{enumerate}
\item $a_3 =0$,
\item $a_3 >0$, $a_1 \leq a_2$ and $a_1^2+a_3^2 < 2a_1a_2$,
\item $a_3 >0$, $a_2 \leq a_1$ and $a_2^2+a_3^2 < 2a_1a_2$,
\item $a_3<0$ and $a_2 \geq \left(\frac{2gk^2+2k+1}{2(k+1)}\right) \cdot a_1$, where 
$k = \lceil \frac{|a_3|/a_1}{1 - |a_3|/a_1} \rceil$ or
\item $a_3<0\; \text{and}\; a_1 \geq \left(\frac{2gl^2+2l+1}{2(l+1)}\right) \cdot a_2$, where 
$l = \lceil \frac{|a_3|/a_2}{1 - |a_3|/a_2} \rceil$.
%\item a_3<0\; \text{and}\; a_2 \geq M\cdot a_1$ $(\text{or}\; a_1>N\cdot a_2)$ for fixed $M\in \mathbb{Q}$ $(N\in \mathbb{Q})$ depending on $g,a_3$ and $a_1$ $(g,a_3\; \text{and}\; a_2)$.
\end{enumerate}
Then $\s(L) \in \mathbb{Q}$.
\end{theorem}

\begin{proof}
$(1)$ Assume $a_3 =0$, then we have $L\equiv_{num} a_1F_1+a_2F_2$. In this case, 
we show that either $L\cdot F_2\leq \sqrt{L^2}$ or $L\cdot F_1\leq \sqrt{L^2}$. This is 
equivalent to show that 
\begin{eqnarray}
\text{either}\quad a_1^2 \leq 2a_1a_2   &\text{or}&  a_2^2 \leq 2a_1a_2,\cr \label{e:ineq}
\Leftrightarrow\quad  \text{either}\quad a_1 \leq 2a_2  &\text{or}& a_2 \leq 2a_1.
\end{eqnarray}
Notice that, when $a_1 > 2a_2$, we get  
\begin{eqnarray*}
2a_1>a_1>2a_2>a_2,
\end{eqnarray*}
implying that the statement \eqref{e:ineq} always holds. 

$(2)$ Let $a_3>0$,  $a_1\leq a_2$  and $a_1^2+a_3^2 <2a_1a_2$. Then, we show 
that $L\cdot F_2 \leq \sqrt{L^2}$. Notice that
\begin{eqnarray} \nonumber
&&L\cdot F_2 \leq \sqrt{L^2} \cr \nonumber
&\Leftrightarrow&(a_1+a_3)^2 \leq L^2 = 2a_1a_2+2a_2a_3+2a_1a_3-a_3^2(2g-2)\cr \nonumber
&\Leftrightarrow& a_1^2+a_3^2+2a_1a_3 \leq 2a_1a_2+2a_2a_3+2a_1a_3-a_3^2(2g-2)\cr \nonumber
&\Leftrightarrow& a_1^2+a_3^2 \leq 2a_1a_2+2a_2a_3-a_3^2(2g-2)\\ \label{eq1}
&\Leftrightarrow& a_1^2+a_3^2 +a_3^2(2g-2) \leq 2a_1a_2+2a_2a_3\\ \label{eq2}
&\Leftrightarrow& a_1^2+a_3^2(2g-1) \leq 2a_1a_2+2a_2a_3. 
\end{eqnarray}
Now, since $L$ is ample, we have
\begin{eqnarray} \nonumber
&& L\cdot \delta = a_1+a_2-a_3(2g-2) >0 \cr \nonumber
&\Rightarrow& 2a_2 > a_1+a_2 > a_3(2g-2)\cr 
&\Rightarrow& a_2 > a_3(g-1) \label{eq3} \\
&\Rightarrow& 2a_2a_3 > a_3^2(2g-2).
\end{eqnarray}
It is easy to see that the equation \eqref{eq1} follows from the hypothesis and the equation \eqref{eq3}.

$(3)$ The proof follows similar to that of (2).

$(4)$ Let $a_3 < 0$ and $a_2 \geq \left(\frac{2gk^2+2k+1}{2(k+1)}\right) a_1$, where 
$k=\lceil \frac{|a_3|/a_1}{1 - |a_3|/a_1} \rceil$. We will show that 
\[L\cdot F_2=a_1+a_3 \leq \sqrt{L^2}.
\] 
It suffices to show that the equation \eqref{eq2} holds. Since $L$ is ample, we get 
\begin{eqnarray}
L\cdot F_1 = a_2+a_3 >0 \quad\text{and}\quad L\cdot F_2 = a_1+a_3 >0.
\end{eqnarray}
This implies that, $a_3$ can at the very least be $-a_1$, i.e., $a_3 > -a_1 > -a_2$. 
Thus, there must exist a positive integer $k$ such that
\begin{eqnarray}
a_3 > -\frac{k}{k+1}a_1, \label{equa8}
\end{eqnarray}
since $-(k/k+1)a_1$ converges to $-a_1$. Choose the least such positive $k$ for which the 
above inequality holds. That is,
\begin{eqnarray*}
k := \lceil \frac{|a_3|/a_1}{1 - |a_3|/a_1} \rceil.
\end{eqnarray*}
Here, $\lceil x \rceil$ represents the least integer greater than or equal to $x$.
We have the following
\begin{eqnarray}
a_1\left(\frac{2gk^2+2k+1}{2(k+1)}\right) &\leq& a_2 ~~\text{~(by hypothesis in (4))}\cr
\Leftrightarrow \quad a_1^2\left(\frac{2gk^2+2k+1}{(k+1)^2}\right) &\leq& 2a_1a_2\left(\frac{1}{k+1}\right)\cr \label{equa7}
\Leftrightarrow a_1^2+ \left(\frac{k}{k+1} \right)^2a_1^2(2g-1) &\leq& 2a_1a_2+2a_2\left(-\frac{k}{k+1} \right)a_1.
\end{eqnarray}
Therefore, the following holds:
\begin{eqnarray*}
a_1^2+a_3^2(2g-1) &<& a_1^2+ \left(\frac{k}{k+1} \right)^2a_1^2(2g-1) \cr
&\leq& 2a_1a_2+2a_2\left(-\frac{k}{k+1} \right)a_1 \cr
&<& 2a_1a_2+2a_2a_3.
\end{eqnarray*}
Where the first and last inequalities hold by \eqref{equa8} and the fact that $a_3<0$, while the second inequality 
follows from \eqref{equa7}. Therefore, inequality \eqref{eq2} holds.

(5) The proof is similar to that of (4).
\end{proof}

\begin{example}
We give an example to show the occurrence of case (4). Let $X = C \times C$ be a surface of general type, where $C$ is a 
smooth curve of genus $g$. Let $L \equiv_{num} a_1F_1 + a_2F_2 + a_3\delta$ be an ample line bundle on $X$. Assume 
that $a_2 > a_1$ as in the hypothesis of (4). Therefore, $a_2  > a_1 > |a_3| > 0$. 

Note that $k = \lceil \frac{|a_3|}{a_1-|a_3|} \rceil,$ and, in general, we have $1 \le k \le |a_3|$. 
When $a_1 \geq 2|a_3|$, we have $k=1$.  Then, the condition on $a_2$ becomes
\[ a_2 \geq \left(\frac{2gk^2+2k+1}{2(k+1)}\right)a_1 =\left(\frac{2g+3}{4}\right)a_1.\] 
So for an ample line bundle $L \equiv_{num} a_1F_1 + a_2F_2 + a_3\delta$ with 
$a_3 < 0$, $a_1 \geq 2|a_3|$, and $a_2 > \left(\frac{2g+3}{4}\right)a_1$, we have $\s(L) \in \mathbb{Q}$. 

 For example, fix $g=2$ and take $a_3 = -10$. Then if $a_1 = 20$, 
we get the least value of $k$ i.e., 1. In this case, we require $a_2 \geq (7/4)a_1 = 35$. But when $a_1 = 11$, we get the highest value of $k$ i.e., 10. So we require $a_2 \geq (421/22)a_1 = 382.72$.
\end{example}

Now we prove the following theorem for $\s(K_X,r)$, where $X=C\times C$ 
as in the above theorem. The primary motivation behind this theorem is \cite{K-A}.

\begin{theorem}\label{Seshadri constant on very general point}
Let $X= C\times C$, where $C$ is a general member of the moduli of smooth curves 
of genus $g \geq 2$. Let $K_X$ be the canonical line bundle on $X$ and 
$r \geq K_X^2 $ be an integer. Then either
\begin{eqnarray*}
\s(X,K_X,r) \geq \sqrt{\frac{r+2}{r+3}} \sqrt{\frac{K_X^2}{r}}
\end{eqnarray*}
or $\s(X,K_X,r)$ is computed by a curve $C_1$ numerically equivalent to $a(F_1+F_2)$ 
(for some $a \in \mathbb{N}$) passing through $r$ very general points with multiplicity one at each point. 
In other words,
\begin{eqnarray*}
\s(X,K_X,r) =\frac{a(K_X\cdot F_1 +K_X\cdot F_2)}{r}.
\end{eqnarray*}
\end{theorem}

\begin{proof}
Suppose
\begin{eqnarray*}
\s(X,K_X,r) &<& \sqrt{\frac{r+2}{r+3}} \sqrt{\frac{K_X^2}{r}}. 
\end{eqnarray*}
Then, there exists an effective curve $C_1 \subset X$ passing through $s \leq r$ very 
general points with multiplicities one each \cite{K-A}, 
such that
\[
\s(X,K_X,r) = \frac{K_X\cdot C_1}{s} < \sqrt{\frac{r+2}{r+3}}\sqrt{\frac{K_X^2}{r}}.
\]
By \cite[Remark 2.4]{K-A}, we get $C_1^2 < s.$ Also, since $C_1$ is a curve in 
$X$ passing through $s\leq r$ very general points with multiplicities $m_1\geq m_2\geq ...\geq m_s>0$, 
then by Xu's lemma \cite{Xu},
\begin{equation}
C_1^2 \geq \sum_{i=1}^{i=s}{m_i^2} - m_s
       \geq s-1. 
\end{equation}
Thus, we have $C_1^2 = s-1$. We will show that $C_1$ is numerically equivalent to $a(F_1+F_2)$ for 
some $a \in \mathbb{N}$. 

\begin{itemize}
\item [$Case$ 1]: $s=1$ 
\end{itemize}
In this case, we have $\s(X,K_X,r) = K_X\cdot C_1 \geq1> \sqrt{\frac{r+2}{r+3}}\sqrt{\frac{K_X^2}{r}}$
since $r \geq K_X^2$. This is a contradiction to our assumption.

\begin{itemize}
\item [$Case$ 2]: $2\leq s\leq r-1$ 
\end{itemize}
Notice that
\begin{eqnarray} \nonumber
\left(\frac{K_X\cdot C_1}{s}\right)^2 & \geq & \left(\frac{r+2}{r+3}\right)\left(\frac{K_X^2}{r}\right),\\  
\Leftrightarrow\; r(r+3)(K_X\cdot C_1)^2 &\geq& s^2(r+2)K_X^2 . \label{equation} 
%\text{By Index Theorem}, (L\cdot C_1)^2 \geq L^2\cdot C_1^2 \geq (s-1)L^2 , \hfill {C_1^2 \geq s-1} , \\
\end{eqnarray}
Using Hodge Index Theorem, we obtain $(K_X\cdot C_1)^2 \geq (s-1)K_X^2$ and 
hence equation \eqref{equation} follows if we prove 
\[
r(r+3)(s-1) \geq s^2(r+2).
\]
This is true for $r\geq 4$. To see this, it is enough to check the inequality at the 
maximal possible value of $s$, i.e., at $s=r-1$:
\begin{eqnarray*}
r(r+3)(r-2) &\geq& (r-1)^2(r+2) \cr
\Leftrightarrow \; r^3+ r^2 -6r &\geq& r^3 -3r +2 \cr
\Leftrightarrow \; r^2 &\geq& 3r + 2. 
\end{eqnarray*}
This holds for $r\geq 4$. By hypothesis $r\geq K_X^2 = 8(g-1)^2 \geq 8$. So we again arrive at a contradiction to our assumption.

%Thus $$\s(L,r) \geq \left(\frac{r+2}{r+3}\right)\left(\frac{L^2}{r}\right)$$
%if $2\leq s \leq r-1$, assuming $r \geq 4.$
\begin{itemize}
\item [$Case$ 3]: $s=r$ 
\end{itemize}
 %we have $C^2 = r-1$ where $C$ is a general curve passing through $r$ general points. 
Notice that, the equation \eqref{equation} follows if we prove $(K_X\cdot C_1)^2 \geq (C_1^2 +\frac{1}{3})K_X^2 = \left(r-\frac{2}{3}\right)K_X^2$, because we have the following  
\begin{eqnarray*}
r(r+3)\left(r-\frac{2}{3}\right) &\geq& r^2(r+2)\cr
\Leftrightarrow\; r^3-\frac{2}{3}r^2 +3r^2-2r &\geq& r^3+2r^2\cr
\Leftrightarrow\; \frac{r^2}{3} &\geq & 2r.
\end{eqnarray*}
However, the last inequality holds for $r\geq 6$. Now to see 
$(K_X\cdot C_1)^2 \geq (C_1^2 +\frac{1}{3})K_X^2$, we start by 
putting $C_1\equiv_{num} a_1F_1+a_2F_2+a_3\delta$ for some 
$a_1,a_2,a_3\in \mathbb{Z}$ and $L := F_1+F_2$. We know that 
$K_X = 2(g-1)L$ \cite{H}, so it is enough to show that
\begin{eqnarray*}
(L\cdot C_1)^2 &\geq& \left(C_1^2 +\frac{1}{3}\right)L^2\cr
\Leftrightarrow (a_1+a_2+2a_3)^2 &\geq& \left(2a_1a_2+2a_2a_3+2a_1a_3-a_3^2(2g-2)+\frac{1}{3}\right)\cdot 2\cr
\Leftrightarrow a_1^2+a_2^2+4a_3^2+2a_1a_2+4a_2a_3+4a_1a_3 &\geq& 4a_1a_2+4a_2a_3+4a_1a_3-4a_3^2(g-1)+\frac{2}{3}\cr
\Leftrightarrow a_1^2+a_2^2+4a_3^2g &\geq& 2a_1a_2 +\frac{2}{3}.
\end{eqnarray*}
This clearly holds when $a_3\neq 0$. In the case $a_3= 0$, we see that the equation $a_1^2+a_2^2 \geq 2a_1a_2 +2/3$ does not hold only when $a := a_1=a_2$. In the latter case, $C_1\equiv_{num} a(F_1+F_2)$ is a curve passing through $r$ points with multiplicity one each such that
\[\s(X,K_X,r) =\frac{a(K_X\cdot F_1 +K_X\cdot F_2)}{r}.
\]
\end{proof}

%\vspace{.4cm}

Now, for a line bundle of the form $L\equiv_{num} aF_1+bF_2$ with $a,b >0$ 
we explicitly compute the Seshadri constants of $L$ at one or two points.

\begin{theorem}
Let $X=C \times C$, where $C$ is a smooth curve of genus $g \geq 2$ and let 
$L \equiv_{num} aF_1+bF_2$ be an ample line bundle on $X$. Then 
$\s(L, x) = \min\{a,b\}$ for every $x \in X$.
\end{theorem}

\begin{proof}
Since a fibre numerically equivalent to $F_1$ and $F_2$ passes through every point $x \in X$, we get 
\begin{eqnarray*}
\s(L,x) &\leq& L\cdot F_1 = b \quad \text{and}\cr
\s(L,x) &\leq& L\cdot F_2 = a \cr
\Rightarrow \quad \s(L,x) &\leq& \text{min}\{a,b\}.
\end{eqnarray*}
Now, let $C_1$ be any curve in $X$ (not numerically equivalent to $F_1$ and $F_2$) passing through $x$ 
with multiplicity $m$. Then, by Bézout's theorem we obtain
\begin{eqnarray*}
C_1\cdot F_i \geq mult_xC_1 \cdot mult_xF_i = m
\end{eqnarray*}
for $i=1$ and $2$. Therefore, notice that 
\begin{eqnarray*}
L\cdot C_1 &=& a(C_1\cdot F_1) +b(C_1\cdot F_2) \cr
&\geq& \min\{a,b\}(m+m)\cr
\Rightarrow \frac{L\cdot C_1}{m} &\geq& 2\min\{a,b\}.
\end{eqnarray*}
Hence, we get $\s(L,x) = \min\{a,b\}.$
\end{proof}

\begin{theorem}
Let $X=C \times C$, where $C$ is a smooth curve of genus $g \geq 2$ and let 
$L \equiv_{num} aF_1+bF_2$ be an ample line bundle on $X$. Then 
\begin{eqnarray*}
\s(L, x_1,x_2) = \begin{cases}
\min\left\{a,\frac{b}{2}\right\},\; \text{if both $x_1$ and $x_2$ lie on a fixed $F_1$,} \cr
\min\left\{\frac{a}{2},b\right\},\; \text{if both $x_1$ and $x_2$ lie on a fixed $F_2$,} \cr
\min\{a,b\},\;\;\; \text{otherwise.}
\end{cases}
\end{eqnarray*}
\end{theorem}
\begin{proof}
Let $C_1$ be a curve not numerically equivalent to $F_1$ and $F_2$ and passing through $x_1$ and $x_2$ 
with multiplicity $m_1$ and $m_2$ respectively. Since there is a fibre numerically equivalent to $F_1$ and 
$F_2$ passing through every point of $X$, by Bézout's theorem we get
 \begin{eqnarray*}
 C_1\cdot F_1 &\geq& mult_{x_1}C_1\cdot mult_{x_1}F_1= m_1,\cr
 C_1\cdot F_1 &\geq& mult_{x_2}C_1\cdot mult_{x_2}F_1=m_2 \quad \text{and} \cr
  C_1\cdot F_2 &\geq& mult_{x_1}C_1\cdot mult_{x_1}F_1= m_1,\cr
 C_1\cdot F_2 &\geq& mult_{x_2}C_1\cdot mult_{x_2}F_1=m_2.
 \end{eqnarray*}
This gives $C_1 \cdot F_1 \geq m$ and $C_1 \cdot F_2 \geq m$ where $m:= \max\{m_1,m_2\}$.
 Now
 \begin{eqnarray*}
L\cdot C_1 &=& a(C_1\cdot F_1) +b(C_1\cdot F_2) \cr
&\geq& \min\{a,b\}(m+m)\cr
\Rightarrow \frac{L\cdot C_1}{m_1+m_2} &\geq& \frac{\min\{a,b\}(2m)}{m_1+m_2} \geq \min\{a,b\}
\end{eqnarray*}
since $2m \geq m_1+m_2$. Now, if both the points $x_1$ and $x_2$ lie either on a fibre $F_1$ 
or on a fibre $F_2$, then we have 
\begin{eqnarray*}
\s(L,x_1,x_2) &\leq& \frac{L\cdot F_1}{1+1} = \frac{b}{2}\; \text{and}\; \s(L,x_1,x_2) \leq \frac{L\cdot F_2}{1} = a,\quad \text{or} \cr
\s(L,x_1,x_2) &\leq& \frac{L\cdot F_2}{1+1} = \frac{a}{2}\; \text{and}\; \s(L,x_1,x_2) \leq \frac{L\cdot F_1}{1} = b \cr
\Rightarrow\; \s(L,x_1,x_2) &\leq& \min\left\{a,\frac{b}{2}\right\} \; \text{or}\; \s(L,x_1,x_2) \leq \min\left\{\frac{a}{2},b\right\}.
\end{eqnarray*}
However, $\min\{a,b\} \geq \min\left\{a,\frac{b}{2}\right\}$ and $\min\{a,b\} \geq \min\left\{\frac{a}{2},b\right\}$. Therefore, we get 
\begin{eqnarray*}
\s(L,x_1,x_2) &=& \min\left\{a,\frac{b}{2}\right\} \; \text{or}\cr
\s(L,x_1,x_2) &=& \min\left\{\frac{a}{2},b\right\}.
\end{eqnarray*}
In case both the points $x_1$ and $x_2$ do not lie on the same fixed fibre, then we get 
\begin{eqnarray*}
\s(L,x_1,x_2) &\leq& \frac{L\cdot F_1}{1} = b\; \text{and} \cr
\s(L,x_1,x_2) &\leq& \frac{L\cdot F_2}{1} = a.\cr
\Rightarrow\; \s(L,x_1,x_2) &\leq& \min\{a,b\}.
\end{eqnarray*}
Hence, we obtain $\s(L,x_1,x_2) = \min\{a,b\}.$
\end{proof}

\begin{remark}
When $X$ is as in the above two theorems, i.e., of the form $C \times C$, the canonical 
divisor $K_X$ of $X$ is given by $p_1^*(K_C) \otimes p_2^*(K_C)$ where $p_1$ and $p_2$ are 
the two natural projections from $C \times C \longrightarrow C$. Since $deg(K_C)$ is $2(g-1)$, 
$K_X$ is numerically equivalent to $2(g-1)(F_1+F_2)$. Hence the above two theorems apply to $K_X$.
\end{remark}

$\bold{Acknowledgements}:$ I would like to thank my advisor Prof. Krishna Hanumanthu 
for his constant encouragement and many useful remarks which were helpful in writing 
this paper. I would also like to thank Prof. Tomasz Szemberg for giving me the idea 
of studying surfaces of general type of the form $C\times C$ and his generous support when 
I visited the Pedagogical University of Cracow. Lastly, I want to thank Prof. D. S. Nagaraj 
for giving many suggestions which improved this paper.

\begin{center}
\large{{\textbf{References}}}
\end{center}

\renewcommand{\section}[2]{}

\end{document}